\documentclass[11pt]{amsart}
\usepackage{amsmath,amsthm,amscd,amssymb, color}
\usepackage{thmtools}
\usepackage[colorlinks, citecolor = blue,backref]{hyperref}
\usepackage[alphabetic,bibtex-style]{amsrefs}
\usepackage[capitalize, nameinlink]{cleveref}
\usepackage[shortlabels]{enumitem}
\usepackage[margin=1.35in]{geometry}
\newcommand{\Q}{\mathbb Q}
\newcommand{\C}{\mathbb C}
\newcommand{\Z}{\mathbb Z}

\renewcommand{\phi}{\varphi}
\newcommand{\eps}{\varepsilon}
\newcommand{\fraka}{\mathfrak a}
\newcommand{\frakd}{\mathfrak d}
\newcommand{\frakp}{\mathfrak p}

\newcommand{\fraku}{\mathfrak u}
\newcommand{\frakt}{\mathfrak t}
\newcommand{\frakq}{\mathfrak q}

\newcommand{\frakm}{\mathfrak m}
\newcommand{\frakn}{\mathfrak n}
\newcommand{\frako}{\mathfrak o}
\newcommand{\calA}{\mathcal A}

\newcommand{\calN}{\mathcal N}

\newcommand{\bmx}{\left( \begin{matrix}}
\newcommand{\emx}{\end{matrix} \right)}
\newcommand{\Cl}{\mathrm{Cl}}

\renewcommand{\mod}{\bmod}

\usepackage{bbm}

\newcommand{\leg}{\overwithdelims ()}

\DeclareMathOperator{\End}{End}

\newtheorem{lem}{Lemma}[section]
\newtheorem{prop}[lem]{Proposition}
\newtheorem{thm}[lem]{Theorem}
\newtheorem{cor}[lem]{Corollary}
\newtheorem{conj}{Conjecture}

\numberwithin{table}{section}

\newtheorem{mthm}{Theorem}
\crefname{mthm}{Theorem}{Theorems}
\newtheorem{mcor}[mthm]{Corollary}
\crefname{mcor}{Corollary}{Corollaries}

\crefname{app}{Appendix}{Appendices}

\theoremstyle{remark}
\newtheorem{rem}[lem]{Remark}

\crefname{mrem}{Remark}{Remarks}

\theoremstyle{definition}

\numberwithin{equation}{section}

\makeatletter
\@for\env:=lem,prop,thm,cor,rem,question,ex,fact\do{%
  \expandafter\gdef\csname theH\env\endcsname{\theHsection.\arabic{lem}}}
\makeatother

\begin{document}

\title{Rationality fields of CM modular forms}

\author{Kimball Martin}
\address{Department of Mathematics $\cdot$ International Research and Education Center, Graduate School of Science, Osaka Metropolitan University, Osaka 558-8585, Japan}
\email{kimball@omu.ac.jp}

\date{\today}

\begin{abstract} We formulate a conjecture on the finitude of rationality fields (i.e., Fourier coefficient fields) of newforms of bounded degree, and prove this for CM forms assuming a generalized Riemann hypothesis.  Then we explicitly determine what quadratic and cubic rationality fields occur for weight 2 CM forms, which is related to classifications of CM abelian varieties of GL(2) type.  Several of these fields do not appear in the existing tables of newforms in the LMFDB.
\end{abstract}


\maketitle


\section{Introduction}


The motivation for this paper begins with the following conjecture.  For a newform $f = \sum a_n q^n$ lying in $S_k(N) = S_k(\Gamma_0(N))$ for some $k$ and $N$, the rationality (or Hecke) field $K_f = \Q((a_n)_n)$ is the number field generated by the Fourier coefficients of $f$.

\begin{conj} \label{conj1} Fix an integer $d \ge 1$.  

\begin{enumerate}[(a)]
\item For a fixed weight $k$, there are only finitely many number fields $K$ of degree $d$ such that $K = K_f$ for some weight $k$ newform $f$ with trivial nebentypus.

\item There are only finitely many number fields $K$ of degree $d$ such that $K = K_f$ for a non-CM newform $f$ with trivial nebentypus and arbitrary weight.
\end{enumerate}
\end{conj}

To our knowledge, this conjecture has not been previously formulated explicitly in the literature, but it is suggested by existing ideas and conjectures.

First, Coleman's conjecture says that, for a fixed $d$, the set of  isomorphism classes of endomorphism algebras $\End_\Q^0(A)$ of $d$-dimensional simple abelian varieties $A/\Q$ is finite (see \cite{BFGR}).  In particular, if $A_f$ denotes the (simple) abelian variety attached to a weight 2 newform $f$ \`a la Shimura, then $\End_\Q^0(A_f) \simeq K_f$ by \cite{ribet:korea}.  Hence Coleman's conjecture implies \cref{conj1}(a) when $k=2$.

Next, \cite[Conjecture B]{me:maeda}, combined with Roberts' conjecture \cite{roberts}, predicts that there are only finitely many newforms of squarefree level (and trivial nebentypus) with a degree $d$ rationality field for fixed large $k$ and $d \ge 1$.  (Possibly $k \ge 6$ suffices.)  The heuristics behind the conjecture also suggest that, for general levels, there are only finitely many quadratic twist classes of weight $k$ non-CM newforms with rationality field having degree $d$ for large $k$.  This would imply \cref{conj1}(a) for large $k$ when restricting to non-CM newforms.  \cref{thm:fin} below will address the CM case.  

The same philosophy suggests that non-CM newforms with fixed degree $d$ do not appear in high weight --- this idea is presented in \cite[Section 2.3]{CM}, and generalizes Roberts' conjecture \cite{roberts} that rational non-CM newforms do not appear in sufficiently large weight.  This should also follow from a uniform-in-the-weight version of the generalized Maeda conjecture (see \cite[Conjecture B+]{me:maeda} for the squarefree level statement) together effective lower bounds on the dimensions of newforms with prescribed local ramified components.  These ideas suggest that \cref{conj1}(a) should imply \cref{conj1}(b).  

From a modularity point of view, \cite{CM} suggests that one reason \cref{conj1}(a) may hold when $k=2$ is that there may only be finitely many related moduli spaces having non-cuspidal rational points for a fixed $d$.  (When $d=1$, this means that only finitely many modular curves $X_0(N)$ have non-cuspidal rational points, which is known; see \cite{mazur,kenku}.)  In higher weight, newforms do not correspond to abelian varieties but motives.  Philosophically it seems reasonable to posit an analogous geometric finiteness statement about motives which would imply \cref{conj1}.

\begin{rem}
We similarly expect that it is also reasonable to believe an analogue of \cref{conj1} for any fixed nebentypus, and note that the methods used here extend to CM forms with nebentypus.
\end{rem}

At present, there is no clear strategy to tackle \cref{conj1}(a), even for $k=2$.  Our first result proves this conjecture when one restricts to CM forms, subject to the Extended Riemann Hypothesis (ERH) for quadratic Dirichlet characters.

\begin{mthm} \label{thm:fin} Fix $k \ge 2$ and $d \ge 1$.  Assuming ERH, there are only finitely many number fields $K$ of degree $d$ such that $K = K_f$ for some weight $k$ newform $f$ with CM and trivial nebentypus.  If $k=2$ and $d\le 6$, or if $(k-1)d = 2^m 3^\epsilon$ with $m \in \mathbb N$ and $\epsilon \in \{ 0, 1 \}$, then this holds unconditionally.
\end{mthm}

Such an $f$ as in the theorem must have CM by an imaginary quadratic field $E$ with (class group) exponent $n \mid (k-1)d$ (see \cref{cor:exp}).
The assumption of ERH guarantees that there are only finitely many imaginary quadratic fields of exponent $n$ \cite{BK}.  The latter assertion is known unconditionally when $n \le 6$, or $n$ is a power of 2, or $n$ is 3 times a power of 2 \cite{HB}.  Then we prove \cref{thm:fin} by showing an analogous finiteness statement for the collection of value fields of bounded degree of Gr\"ossencharacters of a fixed field $E$ (see \cref{cor:fin}). 

\cref{thm:fin}, as well as the proof, is similar in spirit to the main result of \cite{schutt}, which conditionally proves that there are only finitely many twist classes of CM forms $f$ with $K_f = \Q$ in a given weight.

Now it is natural to ask, for a given $k$ and $d$, what rationality fields actually occur?  This question is considered in \cite{CM}, where the focus is on weight 2 non-CM forms, and for $k=2$ it is related to the existence of rational points on suitable Hilbert modular varieties.  However, even for $k=2$ and $d=2$, it is not clear how to even guess which rationality fields might occur.  (The only option for $d=1$ is $K_f=\Q$, and it is easy to see it occurs for each $k$.)

Under the assumption of ERH, and in some cases a weaker hypothesis on non-existence of Siegel zeroes, \cite{EKN} enumerates all $E$ with exponent at most 5 or exponent 8.  It is unconditionally known that there is at most one other $E$ for exponent 2, 4 or 8 \cite{weinberger,EKN}.  By determining the possible rationality fields for a given $E$ in the conditionally complete list of exponent $\le 3$ fields, we deduce the following.  

\begin{mthm} \label{thm2} Assuming ERH, the complete set of degree $2$ and $3$ rationality fields $K$ of weight $2$ CM forms with trivial nebentypus are the $23$ quadratic fields and $18$ cubic fields with discriminant $\Delta_K$ listed in \cref{tab:deg2K,tab:deg3K}, respectively.
\end{mthm}

In \cref{tab:deg2K,tab:deg3K}, along with each $\Delta_K$, we list the discriminants $\Delta_E$ of all $E$ such that $K=K_f$ for some newform $f$ with CM by $E$ (subject to ERH).  In all cases listed, the degree and discriminant of $K$ determine $K$ up to isomorphism. The first two cubic fields $K$ appearing in \cref{tab:deg3K} are $\Q(\zeta_7)^+$ and $\Q(\zeta_9)^+$. 

In \cref{tab:deg2K}, the $E$'s occurring all have class number 1 or 2.  When $h_E = 2$ the corresponding $K$ is the real quadratic field such that $L=EK$ is the Hilbert class field of $E$, though we do not have a non-computational proof of this fact.  In \cref{tab:deg3K} the $E$'s which appear are the imaginary quadratic fields of class number 3 and the two class number one fields $\Q(\sqrt{-3})$ and $\Q(\sqrt{-7})$.  But when $h_E = 3$, the Hilbert class field of $E$ has no nontrivial totally real subfields so cannot contain $K$. 

Our strategy for \cref{thm2} is, for each $E$ with exponent dividing $d$, to first determine a potential list of degree $d$ rationality fields $K$ using necessary conditions on relevant Gr\"ossencharacters of $E$.  Then for each such $K$ we either construct a Gr\"ossencharacter yielding this $K$ or prove this is impossible by contradiction.  It is convenient to consider Gr\"ossencharacters in 3 cases: (i) when $E$ has class number 1 (\cref{sec:h1}), (ii) when an associated modulus character is quadratic and $h_E > 1$ (\cref{sec:quadmod}), and (iii) when the associated modulus character has higher order and $h_E > 1$ (\cref{sec:last}).  

Let us discuss how the rationality fields in \cref{thm2} compare to known tables of low degree rationality fields.  The LMFDB\footnote{\url{https://www.lmfdb.org/} (last accessed October 12, 2026)} tabulates low degree rationality fields for all weight 2 newforms of level at most 50000, thanks to the recent work \cite{boo}.

There are 44 real quadratic fields $K$ which occur in this range (for trivial nebentypus).  The first missing fundamental discriminant is 77, and the largest occurring discriminant is $268$. The two largest discriminants from \cref{tab:deg2K}, namely $489$ and $652$,  do not occur in the LMFDB data but all others do.  Our constructions in \cref{sec:h1d2} produce CM forms $f$ with these rationality fields in levels $106276 = 2^2 \cdot 163^2$ and $239121 = 3^2 \cdot 163^2$.  We also note that there is exactly one newform $f$ in the LMFDB with rationality field $K = \Q(\sqrt{43})$ (i.e., $\Delta_K = 172$), and similarly for $K = \Q(\sqrt{67})$ (i.e., $\Delta_K = 268$).  These forms occur in level $3^2 \cdot 43^2 = 16641$ and $3^2 \cdot 67^2 = 40401$, and our constructions in \cref{sec:h1d2} show that these are CM forms.  (The CM data in the LMFDB is incomplete for newforms of level over 10000.)  The second largest discriminant appearing in the LMFDB is $241$; it occurs in level 18564 and does not come from a CM form.

There are 161 cubic rationality fields in the LMFDB data, with the largest field having discriminant 20613.  Six of the rationality fields $K$ from \cref{tab:deg3K} do not appear in the LMFDB data, namely those $K$ with $\Delta_K = 7641$ or $\Delta_K \ge 10233$.  CM forms with these 6 rationality fields occur in level $\Delta_E^2$, which ranges from $80089$ for $\Delta_K = 7641$ to $822649$ for $\Delta_K = 24489$ (see \cref{sec:qmce3}).

\begin{rem}
We expect that with more work \cref{thm2} could be extended to $d = 4, 5, 8$ for weight 2 newforms.  To get complete lists of rationality fields of CM forms for other $d$ or $k$, even subject to ERH, would require extending \cite{EKN} to other exponents.  

One reason we only attempted to treat $d=2, 3$ is that for $d=4$ one needs to analyze all exponent 4 fields.  Assuming ERH, there are many more fields of exponent 4 (203 fields) than exponent 2 (56 fields) or exponent 3 (17 fields).  Furthermore, as will be clear from the proof, there are fewer situations to consider when $(k-1)d$ is prime.
\end{rem}

\begin{table}
\begin{tabular}{c|c  @{\hspace{0.3cm}} || @{\hspace{0.3cm}} c|c @{\hspace{0.3cm}} || @{\hspace{0.3cm}} c|c}
$\Delta_K$ & $\Delta_E$ &  $\Delta_K$ & $\Delta_E$ & $\Delta_K$ & $\Delta_E$  \\
\hline
5   & $-15,-20,-35,-40,-115,-235$ & 29  & $-232$   & 89  & $-267$  \\
8   & $-4,-8,-24,-88$             & 33  & $-11$    	& 129 & $-43$ \\
12  & $-3,-4$                     & 37  & $-148$   	& 172 & $-43$  \\
13  & $-52, -91, -403$                       & 41  & $-123$   & 201 & $-67$  \\
17  & $-51,-187$                  & 44  & $-11$    	& 268 & $-67$  \\
21  & $-7$                         & 57  & $-19$     	 & 489 & $-163$ \\
24  & $-8$                         & 61  & $-427$ 		 & 652 & $-163$  \\
28  & $-7$                        & 76  & $-19$  		  \\
\end{tabular}
\caption{Quadratic rationality fields $K$ of weight 2 forms with CM by $E$}
\label{tab:deg2K}
\end{table}

\begin{table}
\begin{tabular}{c|c  @{\hspace{0.3cm}} || @{\hspace{0.3cm}} c|c @{\hspace{0.3cm}} || @{\hspace{0.3cm}} c|c}
$\Delta_K$ & $\Delta_E$ &  $\Delta_K$ & $\Delta_E$ & $\Delta_K$ & $\Delta_E$  \\
\hline
49& $-7$ & 1593& $-59$ & 8289& $-307$ \\ 
81& $-3$ & 1929& $-643$ & 10233& $-379$ \\ 
321& $-107$ & 2241& $-83$ & 13473& $-499$ \\ 
621& $-23$ & 3753& $-139$ & 14769& $-547$ \\ 
837& $-31$ & 5697& $-211$ & 23841& $-883$ \\ 
993& $-331$ & 7641& $-283$ & 24489& $-907$ \\ 
\end{tabular}
\caption{Cubic rationality fields $K$ of weight 2 forms with CM by $E$}
\label{tab:deg3K}
\end{table}

Finally, we note the following consequence on the existence of abelian varieties with special properties.  Recall from \cite{shimura:nagoya} that, given a newform $f \in S_2(N)$ with CM by $E$, 
the associated abelian variety $A_f$ is $\overline{\Q}$-isogenous to $\mathcal E^d$ for some elliptic curve $\mathcal E/\overline{\Q}$ with CM by $E$, where $d = [K_f:\Q]$.   

\begin{mcor}
For each pair $(E,K)$ occurring in \cref{tab:deg2K} and \cref{tab:deg3K}, there is a simple abelian variety $A/\Q$ of dimension $d = [K:\Q]$ with $\Q$-endomorphism algebra $K$ such that $A$ is $\overline{\Q}$-isogenous to a power of an elliptic curve $\mathcal E/\overline{\Q}$ with CM by $E$.
\end{mcor}

In particular, this gives some constructive information about the question: for which totally real fields $K$ does there exist a $d$-dimensional simple abelian variety over $\Q$ with real multiplication (defined over $\Q$) by an order in $K$?

\begin{rem} A related problem is, for a number field $F$, to classify CM elliptic curves $\mathcal E$ with minimal field of definition $F$, i.e., $F = \Q(j(\mathcal E))$.  For $F$ quadratic, there exist such $\mathcal E$ only for 14 real quadratic fields, namely the 11 real quadratic fields with discriminant $\Delta < 40$, $\Q(\sqrt{41})$, $\Q(\sqrt{61})$ and $\Q(\sqrt{89})$.  Curves for such fields are given in \cite{GLY}.

Note that these 14 real quadratic fields $F$ are a proper subset of the 23 real quadratic fields $K$ in \cref{tab:deg2K}.  When $\mathcal E/K$ has CM by the maximal order of a class number two field $E$, this follows from the coincidence in \cref{tab:deg2K} between $EK$ and the Hilbert class field mentioned above.
In general, one might wonder to what extent there is overlap between the CM newforms $f$ with quadratic rationality field $K$ and CM elliptic curves $\mathcal E$ with minimal field of definition $F=K$.

 According to the LMFDB, the first 3 weight 2 CM newforms $f$ with quadratic rationality fields when ordered by level (LMFDB labels \verb+225.2.a.f+, \verb+256.2.a.e+, \verb+441.2.a.h+)  have respective rationality fields $K = \Q(\sqrt 5)$, $\Q(\sqrt 2)$ and $\Q(\sqrt 7)$ and all correspond to elliptic curves $\mathcal E$ with minimal field of definition $F=K$.  Hence there is some overlap, but this pattern does not persist --- the next CM form $f$ (LMFDB label \verb+484.2.a.d+) has rationality field $K = \Q(\sqrt{33})$, but one can check that its base change to $K$ does not have rational Fourier coefficients, so $A_f$ is not $K$-isogenous to an elliptic curve defined over $F=K$.  Conversely, there are CM elliptic curves over quadratic fields $F$ which do not correspond to modular forms over $\Q$, e.g., the elliptic curve with LMFDB label \verb+4091.1-k1+ is defined over $F=\Q(\sqrt 5)$ and the associated Hilbert modular form is not a base change from $\Q$.
\end{rem}

\subsection*{Acknowledgments}
I thank \'Alvaro Lozano-Robledo and Sam Frengley for comments.  
Explicit calculations with quadratic fields were carried out with Sage.
The author was partially supported by JSPS KAKENHI Grant Number 25K23342.

\subsection*{AI usage declaration}
All mathematics and text are my own.  I did use Claude to help proofread this paper, which led to minor corrections.


\section{CM forms}


In this section, we set some notation, recall basic facts about Gr\"ossencharacters and CM forms, and relate the rationality fields of CM forms to the fields of values of  Gr\"ossencharacters.

\subsection{Notation}

For a field $F$, let $\mu_F$ be the group of roots of unity in $F$.  For a finite-degree  field extension $K/F$, denote by $\calN = \calN_{K/F}$ the norm map from $K$ to $F$.  
If $F$ is a number field,  $\Delta_F$ denotes its discriminant, $\frako_F$ its ring of integers, $\Cl(F)$ its ideal class group, and $h_F = \# \Cl(F)$.

For $n\ge 1$, let $\zeta_n$ be a primitive $n$-th root of unity, and write $\mu_n = \langle \zeta_n \rangle$.  We denote the cyclic group of order $n$ by $C_n$.

Throughout, $E/\Q$ will be an imaginary quadratic field, $p$ will denote a prime in $\mathbb N$ and $\frakp$ will denote a prime of $E$ above $p$.  Further, $\frakm$ will denote a nonzero integral ideal of $E$.

For $\alpha \in E^\times$, write $\alpha \equiv 1\mod \frakm$ if $v_\frakp(\alpha -1) \ge v_\frakp(\frakm)$ for each prime $\frakp \mid \frakm$.  (Some authors write $\alpha \equiv 1 \text{ mod}^\times \, \frakm$.)  Let $E^\frakm = \{ \alpha \in E^\times : v_\frakp(\alpha) = 0 \text{ for all } \frakp \mid \frakm \}$ and $E_1^\frakm = \{ \alpha \in E^\times : \alpha \equiv 1 \mod \frakm \}$.
Denote by $J^\frakm = J^\frakm(E)$  the group of (invertible) fractional ideals of $E$ relatively prime to $\frakm$, $P^\frakm = P^\frakm(E)$ the subgroup of principal ideals, and $P_1^\frakm = P_1^\frakm(E) = \{ \alpha \frako_E : \alpha \in E_1^\frakm \}$ the subgroup of principal ideals which are $1 \mod \frakm$.  The associated ray class group is $\Cl^\frakm(E) = J^\frakm(E)/P_1^\frakm(E)$.

\subsection{Gr\"ossencharacters and CM forms}
\label{sec:gros-def}

Fix an imaginary quadratic field $E/\Q$, and a nonzero integral ideal $\frakm$ of $E$.  Let $\psi$ be a Gr\"ossencharacter of $E$ with conductor $\frakm$ and weight $\ell \in \Z_{\ge 0}$, which means $\psi: J^\frakm \to \C^\times$ is a character such that
\[ \psi(\alpha \frako_E) = \alpha^\ell \quad \text{if } \alpha \equiv 1 \bmod \frakm. \]
Given $\frakm$ and $\ell$, there exists such a $\psi$ if and only if $\zeta^\ell = 1$ for each $\zeta \in \frako_E^\times \cap E_1^\frakm$.  Note that $\frako_E^\times \cap E_1^\frakm$ is only nontrivial if (i) $2 \in \frakm$ or (ii) $\Delta_E = -3$ and $\calN (\frakm) \in \{ 1, 3 \}$.

The map $E^\frakm/E_1^\frakm \simeq (\frako_E/\frakm)^\times \to P^\frakm/P_1^\frakm$ is surjective, and its kernel is isomorphic to $\frako_E^\times/(\frako_E^\times \cap E_1^\frakm)$.  We often identify $E^\frakm/E_1^\frakm$ with $(\frako_E/\frakm)^\times$.  Then $\psi$ gives rise to a character $\eta_\psi : (\frako_E/\frakm)^\times \to \C^\times$ via
\[  \eta_\psi(\alpha) = \alpha^{-\ell} \psi(\alpha \frako_E), \]
which we call the \emph{modulus character} of $\psi$.
Replacing $\alpha$ by $u \alpha$ shows that $\eta_\psi(u) = u^{-\ell}$ for any $u \in \frako_E^\times = \mu_E$.  

Extend $\psi$ to be $0$ on integral ideals not coprime to $\frakm$.
Then the  following result can be found in \cite[Lemma 3]{shimura:nagoya} and \cite[Section 3]{ribet}.
\begin{thm}[Hecke, Shimura] \label{thm:heckeshimura}
Define
\[ f_\psi = \sum_{\fraka} \psi(\fraka) q^{\calN \fraka}, \] 
where $\fraka$ runs over integral ideals of $E$.  Let $\eta_\psi^\Z$ be the Dirichlet character of $\Z$ mod $M = \calN \frakm$ given by composing $\eta_\psi$ with the natural map $\Z/M\Z \to \frako_E/\frakm$.   Then
\[ f_\psi \in S_{\ell + 1}(|\Delta_E| M, \chi_E \eta_\psi^\Z). \]
If $\psi$ is primitive (see \cref{sec:cond}), then $f_\psi$ is a newform.
\end{thm}

In particular, we have the property $f_\psi \otimes \chi_E = f_\psi$, which means that $f_\psi$ has complex multiplication (CM) by $E$.  Further, any newform with CM by $E$ is of the form $f_\psi$ for some such $\psi$ (\cite[Theorem 4.5]{ribet}).

We are most interested in the case that $\ell$ is odd and $\eta_\psi^\Z = \chi_E$ (as Dirichlet characters mod $|\Delta_E| M$), which corresponds to $f_\psi$ having even weight and trivial nebentypus.  This will imply that $\psi(p\frako_E) = {\Delta_E \leg p} p^\ell$ when $p \nmid M$.  Further, the conductor of $\chi_E$ will have to divide that of $\eta_\psi^\Z$, which means that $\Delta_E \mid M$ (see \cref{lem:mincond} for a more precise condition).  The condition $\eta_\psi(-1) = \chi_E(-1) = -1$ implies that  $-1 \not \in \frako_E^\times \cap E_1^\frakm$, and thus $2 \in \frakm$ is impossible.

\subsection{Rationality fields and value fields}

Now we connect rationality fields of CM forms (with trivial nebentypus) with the values of Gr\"ossencharacters (of odd weight $\ell \ge 1$).  By the \emph{value field} $L_\psi$ of a Gr\"ossencharacter $\psi$, we mean the subfield of $\C$ generated by $\text{Im}\, \psi$.  Since $1 + \frakm \subset E_1^\frakm$, the set $\psi(P_1^\frakm) \subset E^\times$ generates $E$.  Thus $L_\psi$ is a number field containing $E$, as it is generated by $\psi(P_1^\frakm)$ and the values of $\psi$ on a set of representatives for the ray class group $\Cl^\frakm(E)$.

\begin{lem} \label{lem22} Let $f = f_\psi \in S_k(N)$ with rationality field $K = K_f$, and let $L=L_\psi$ be the value field of $\psi$.  Then $L = EK$. 
\end{lem}

\begin{proof}
From \cref{thm:heckeshimura}, $a_p(f) = \psi(\frakp) + \psi(\bar \frakp)$ if $(p) = \frakp \bar \frakp$ is split in $E$,  $a_{p^2}(f) = \psi((p))$ if $p$ is inert and $a_{p}(f) = \psi(\frakp)$ if $(p) = \frakp^2$ is ramified in $E$.  If $(p) = \frakp^2$ is ramified, then $\psi(\frakp) = 0$. If $p$ is inert, then $\psi((p)) = \eta_\psi(p) p^\ell = \chi_E(p) p^\ell \in \Z$.  
Hence $L$ is generated by the values $\psi(\frakp)$, and $K$ is generated by the values $\psi(\frakp) + \psi(\bar \frakp)$, where in both cases $p$ ranges over primes split in $E$ and $(p) = \frakp \bar \frakp$.

Assume now $(p) = \frakp \bar \frakp$ is split in $E$.
Since $\psi(\bar \frakp) = \psi(p \frako_E)/\psi(\frakp) = p^\ell/\psi(\frakp)$, we see that $\beta = \psi(\frakp)$ satisfies a quadratic polynomial over $K$, and thus $\beta$ lies in a quadratic extension of $K$.

Say $\frakp$ has order $n$ in $\Cl(E)$, and write $\frakp^n = \pi \frako_E$ for some irreducible $\pi \in E$.  Let $r$ be the multiplicative order of $\pi$ mod $\frakm$.  Then $\beta^{nr} = \psi(\frakp^{n})^r = \pi^{r \ell}$, so $\beta^n = \psi(\frakp^n) = \zeta \pi^\ell$ for some $\zeta \in \mu_r$.  Now any positive power of $\pi$ lies in $E - \Q$ (or else $\frakp$ would ramify over $p$), so some positive power of $\beta$ generates $E$.  However, since $\beta$ was quadratic over $K$ and $E \subset K(\beta)$, we must have that $K(\beta) = EK$.
 This implies $L \subseteq E K$.

For the converse, note that $L \supseteq K$ is clear.  Now by the previous paragraph, some power of $\beta = \psi(\frakp)$ generates $E$ and thus $L \supseteq EK$.
\end{proof}

Thus we can characterize the rationality field $K$ of $f_\psi \in S_k(N)$ as the unique totally real index 2 subfield of the value field $L$ of $\psi$.  Hence to study rationality fields of CM forms with trivial nebentypus, it suffices to study value fields of Gr\"ossencharacters.


\section{Value fields of Gr\"ossencharacters}


Here we prove some basic properties about value fields of Gr\"ossencharacters and deduce \cref{thm:fin}.
Let $\psi$ be a Gr\"ossencharacter of $E$ of conductor $\frakm$ and weight $\ell$.  Denote by $L = L_\psi$ its value field, and set $d = [L:E]$.

\subsection{Ideal class behavior in value fields}

Note that if $\alpha \frako_E \in P^\frakm$ has order $r$ mod $P_1^\frakm$, then $\psi(\alpha \frako_E) = \zeta \alpha^\ell$ for some $\zeta \in \mu_r$.  Since $\alpha^\ell \in L$, this means that in fact $\zeta \in \mu_L$.

\begin{prop} \label{prop:aL}  For any nonzero ideal $\fraka \subset \frako_E$, the $L$-ideal class of $\fraka \frako_L$ has order dividing $\ell$ in $\Cl(L)$.
\end{prop}

When $\ell = 1$, this means that $L$ is a principalization (or capitulation) field for $E$.

\begin{proof} Since $\Cl(E)$ is generated by ideals coprime to $\frakm$, we may replace $\fraka$ by an equivalent ideal to assume that $\fraka$ is coprime to $\frakm$.  

Consider $\fraka \in J^\frakm$. Say $\fraka$ has order $r$ in $\Cl(E)$, and write $\fraka^r = \alpha \frako_E$, where $\alpha \in E^\times$.    
Since $\psi(\fraka^r) = \eta_\psi(\alpha)\alpha^\ell \in \alpha^\ell \frako_L^\times$, we have
\[ (\psi(\fraka) \frako_L)^r = \psi(\fraka^r) \frako_L = \alpha^\ell \frako_L = (\alpha \frako_L)^\ell = (\fraka \frako_L)^{r \ell}. \]
By unique factorization of ideals, we have $(\fraka \frako_L)^\ell =\psi(\fraka) \frako_L$, which is  principal.
\end{proof}

\begin{cor} \label{cor:exp} The exponent of $\Cl(E)$ divides $\ell d$.
\end{cor}

\begin{proof}
For any ideal $\fraka \subset \frako_E$, note that $\calN_{L/E}(\fraka \frako_L) = \fraka^d$.  Now suppose $\fraka$ is nonzero, and thus $\fraka^\ell \frako_L = \beta \frako_L$ for some $\beta \in \frako_L$ by \cref{prop:aL}.  Hence $\fraka^{\ell d} = \calN_{L/E}(\fraka^\ell \frako_L) = \calN_{L/E}(\beta) \frako_E$ is principal.
\end{proof}

The above results generalize the case of $L=E$ in \cite{schutt}.  They allow us to reduce the problem of classifying degree $d$ rationality fields of CM forms in $S_k(N)$ to considering those with CM by a field $E$ with (class group) exponent dividing $(k-1)d$.

\subsection{Restrictions on value fields}

Now for fixed $E$ and weight $\ell$, we obtain necessary conditions on the possible value fields $L$.

\begin{lem} \label{lem:L0} Let $L_0 \subset L$ be the subfield generated by the values of $\psi$ on principal ideals in $P^\frakm$.  Then $L_0 = E(\zeta_r)$, where $r$ is the order of $\eta_\psi$.
\end{lem}

\begin{proof} As explained before \cref{lem22}, $\psi(P^\frakm_1)$ generates $E$.  Thus $E \subset L_0$.  For any principal ideal in $P^\frakm$, we may scale it by an element of $P^\frakm_1$ so that it is represented by some $\alpha \in \frako_E$.  Then $\eta_\psi(\alpha) = \psi(\alpha\frako_E)/\alpha^\ell \in L_0$.  This implies $E(\zeta_r) \subset L_0$. 

On the other hand, $\psi(P^\frakm) \subset \eta_{\psi}((\frako_E/\frakm)^\times)E$, whence $L_0 \subset E(\zeta_r)$.
\end{proof}

For $\alpha \in E^\times$, let $\alpha^{1/n}$ denote some $n$-th root of $\alpha$ in the algebraic closure of $E$.

\begin{prop} \label{prop:vf}
Let $\frakt_1, \dots, \frakt_g$ be representatives for a set of generators for $\Cl(E)$.  Let $n_i$ be the order of $\frakt_i$ in $\Cl(E)$ and write $\frakt_i^{n_i} = \theta_i \frako_E$.  Let $r$ be the order of $\eta_\psi$.  Then 
\[ L = L_\psi = E(\zeta_r, (u_1 \theta_1^\ell)^{1/n_1}, \dots, (u_g \theta_g^\ell)^{1/n_g}) \] for some $u_1, \dots, u_g \in \mu_r$ (and some choices of $n_i$-th roots of $u_i \theta_i^\ell$).  In particular, $L \subset \tilde E_{r,\ell} := E(\zeta_{rn}, \theta_1^{\ell/n_1}, \dots, \theta_g^{\ell/n_g})$ where $n$ is the exponent of $\Cl(E)$.
\end{prop}

Note that $\tilde E_{r,\ell}$ really only depends on $E$, $r$ and $\ell$, and not on the choices of $\frakt_i$ or $\theta_i$ or the $n_i$-th roots.

\begin{proof}
For $1 \le i \le g$, select $\fraka_i = \gamma_i \frakt_i \in J^\frakm$ for some $\gamma_i \in E^\times$.  Then $\fraka_i^{n_i} = \alpha_i \frako_E$, where $\alpha_i = \gamma_i^{n_i} \theta_i$.  From \cref{lem:L0}, the values of $\psi$ on principal ideals generate $L_0 = E(\zeta_r)$.  Hence $L_\psi = L_0(\psi(\fraka_1), \dots, \psi(\fraka_g))$.
Note that $\psi(\fraka_i)$ is an $n_i$-th root of $\eta_\psi(\alpha_i)\alpha_i^\ell = \eta_\psi(\alpha_i)\theta_i^\ell \gamma_i^{n_i \ell}$.  Hence $L_0(\psi(\fraka_i)) = L_0((\eta_\psi(\alpha_i) \theta_i^\ell)^{1/n_i})$.
\end{proof}

\begin{rem} \label{rem:Leta}
 Fix $E, \ell$ and $r$ as above.  From the proof we see that $L_\psi$ only depends on $\eta_\psi$ and the choice of certain $n_i$-th roots of unity.  If $\Cl(E)$ is cyclic, there is only one $n_i$ and any choice of an $n_i$-th root will give the same $L_\psi$ up to field isomorphism.  If $n_i \mid r$, then the choices of roots of unity do not affect the value field since $\mu_r \subset L_\psi$.  Hence if either the exponent $n \mid r$ or $\Cl(E)$ is cyclic, then $\eta_\psi$ determines $L_\psi$ up to isomorphism. 
\end{rem}

Once we bound the degree of the value field $L_\psi$, there are finitely many possibilities for the order of $\eta_\psi$ (without restriction on $\ell$ or $\frakm$) by \cref{lem:L0}.  Noting that the field $\tilde E_{r,\ell}$ in the proposition does not depend upon the conductor, and only depends on $\ell \mod n$, we deduce the following.

\begin{cor} \label{cor:fin}
Given $E$ and $d \ge 1$, there are finitely many extensions $L/E$ of degree $d$ which are value fields of Gr\"ossencharacters of $E$ (for arbitrary weights and conductors).
\end{cor}

This, together with \cref{cor:exp} and the finiteness results about imaginary quadratic fields with given exponent mentioned in the Introduction, proves \cref{thm:fin}.

\section{Minimal conductors} \label{sec:cond}

In order to classify CM newforms with degree $d$ rationality field (assuming a classification of $E$ with given exponent), we would like to be able to do the following.
Given $E$, $d$ and $\ell$, classify: (i) degree $2d$ value fields $L$ of weight $\ell$ Gr\"ossencharacters $\psi$ such that $\eta_\psi^\Z = \chi_E$, and (ii) all Gr\"ossencharacters $\psi$ with a given value field $L$ such that $\eta_\psi^\Z = \chi_E$.  While (i) is our goal in this paper, we will say some things about (ii).  In fact a large part of (ii) is determining what conductors $\frakm$ are possible, and this is useful for (i).  Problem (ii) was taken up in \cite{schutt} when $L=E$, which corresponds to $K=\Q$.  In this section we establish some results about what conductors need to be considered.

Let $\psi$ be a Gr\"ossencharacter of $E$ with weight $\ell$ and conductor $\frakm$.  

We say that $\psi$ (resp.\ $\eta_\psi$) is primitive of conductor $\frakm$ if there is no integral ideal $\frakm'$ properly dividing $\frakm$ such that $\psi$ extends to a Gr\"ossencharacter of conductor $\frakm'$ (resp.\ $\eta_\psi$ factors through $(\frako_E/\frakm')^\times$).  For our classification problems, it clearly suffices to restrict to primitive $\psi$.  The following says that $\psi$ is primitive if and only if $\eta_\psi$ is primitive.

\begin{lem} \label{lem:psi-extend}
Suppose $\frakm'$ is a nonzero integral ideal of $E$ such that $\frakm' \mid \frakm$ and that $\eta_\psi$ factors through $(\frako_E/\frakm')^\times$.  Then $\psi$ is the restriction to $J^\frakm$ of a Gr\"ossencharacter $\psi'$ with conductor $\frakm'$ such that $\eta_{\psi'} = \eta_\psi$ and $L_{\psi'} = L_{\psi}$.
\end{lem}

\begin{proof}
By induction, it suffices to prove this in the case that $\frakm = \frakp \frakm'$.  If $\frakp \mid \frakm'$, then $J^{\frakm'} = J^\frakm$ and one only needs to check that $\psi(\alpha) = \alpha^\ell$ when $\alpha \equiv 1 \mod \frakm'$.  But this is immediate from $\eta_\psi$ being trivial on such $\alpha$.  

Assume $\frakm = \frakp \frakm'$ where $\frakp \nmid \frakm'$.
To extend $\psi$ to a character $\psi'$ of $J^{\frakm'}$, we only need to specify $\psi'(\frakp)$.  Write $\frakp = \alpha \frakq$, where $\alpha \in E^\times$ and $\frakq \in J^\frakm$.  Then $\alpha$ is coprime to $\frakm'$, and we can define $\psi'(\frakp) = \eta_\psi(\alpha) \alpha^\ell \psi(\frakq)$.  

To see this is independent of the choice of $\alpha$ and $\frakq$, consider another representation $\frakp = \alpha' \frakq'$ where $\frakq' \in J^\frakm$.  Then $\frakq (\frakq')^{-1} = \alpha' \alpha^{-1} \frako_E \in J^\frakm$.  From this one deduces
 $\eta_\psi(\alpha') (\alpha')^\ell \psi(\frakq') = \eta_\psi(\alpha) \alpha^\ell \psi(\frakq) = \psi'(\frakp)$.

Now $\eta_\psi$ being trivial on $E_1^{\frakm'}$ means that $\psi'$ has conductor $\frakm'$.  The assertions on $\eta_{\psi'}$ and $L_{\psi'}$ are also clear.
\end{proof}

Given a Dirichlet character $\chi$ of conductor $Q$, we define the twist of $\psi$ to be the primitive weight $\ell$ Gr\"ossencharacter of $E$ such that $(\psi \otimes \chi)(\fraka) = \chi(\calN(\fraka)) \psi(\fraka)$ for $\fraka \in J^{\frakm Q}$.
A priori its conductor is just a divisor of $\frakm Q$; however, $\psi \otimes \chi$ may have conductor which is a proper divisor of $\frakm$. 

Suppose $\chi$ is a quadratic character.  Then the quadratic twist $\psi \otimes \chi$ has the same value field as $\psi$.  (More generally, if $\chi$ has image contained in $\mu_L$, then $L_{\psi \otimes \chi} \subset L_\psi$, but they are not in general equal.)  Further $\eta_{\psi \otimes \chi}^\Z = \eta_\psi^\Z$.  Consequently in our study of value fields, we may reduce to considering primitive Gr\"ossencharacters which have minimal conductor among the family of all quadratic twists $\psi \otimes \chi$.

\begin{rem}
By only working with newforms of trivial nebentypus or newforms with a fixed rationality field $K$, we cannot always restrict to considering minimal CM forms or Gr\"ossencharacters.  For instance, there is a newform $f \in S_2(15^2)$ (labelled \verb+225.2.a.f+ in the LMFDB) with rationality field $\Q(\sqrt 5)$ and CM by $\Q(\sqrt{-15})$, but it is not minimal: it is an order 4 twist of a level 45 form with character ${5 \leg \cdot}$ that has rationality field $\Q(\sqrt{-5})$.
 \end{rem}
 
 Now let us restrict to the setting that $\eta_\psi^\Z = \chi_E$.  We have the following minimal divisibility condition on the conductor.  For this and subsequent results, we will need some facts about the structure of modular multiplicative groups $(\frako_E/\frakp^e)^\times$.  These facts are well known for $p$ odd, and are recalled in \cref{sec:append} below for $p = 2$. 
  
 \begin{lem} \label{lem:mincond}
  If $\eta_\psi^\Z = \chi_E$ then $\frakd_E \mid \frakm$, where $\frakd_E$ is defined by
 \[ \frakd_E := \begin{cases}
\sqrt{\Delta_E} \frako_E & \text{if } 2 \nmid \Delta_E \\
\frakp_2 \sqrt{\Delta_E} \frako_E &  \text{if } 4 \parallel \Delta_E \\
2 \sqrt{\Delta_E} \frako_E &  \text{if } 8 \parallel \Delta_E,
\end{cases} \] 
and $\frakp_2$ is the prime above $2$ in $E$ when $2 \mid \Delta_E$.
 \end{lem}
 
 \begin{proof}
 Write $\Delta_E = -2^e q_1 \dots q_t$, where the $q_i$'s are odd primes and $e \in \{ 0, 2, 3 \}$.  Then $(\Z/\Delta_E\Z)^\times \simeq (\Z/2^e\Z)^\times \times \prod (\Z/q_i\Z)^\times$.  
 Since $\chi_E$ is a primitive Dirichlet character of conductor $|\Delta_E|$, 
 we must have that $(\Z/p^{v_p(\Delta_E)}\Z)^\times$ injects into $(\frako_E/\frakp^{v_\frakp(\frakm)})^\times$ for all $p \mid \Delta_E$ (where, as usual, $\frakp$ denotes the prime above $p$).  In particular $\sqrt{\Delta_E} \mid \frakm$.
 
 Suppose now that $\Delta_E$ is even.  Since $(\Z/4\Z)^\times$ (resp.\ $(\Z/8\Z)^\times$) embeds in $(\frako_E/\frakp_2^j)^\times$ if and only if $j \ge 3$ (resp.\ $j \ge 5$), we must have that $\frakp_2^3 \mid \frakm$ when $4 \parallel \Delta_E$ and $\frakp_2^5 \mid \frakm$ when $8 \parallel \Delta_E$.
 Thus in all cases $\frakd_E \mid \frakm$.
  \end{proof}

Consequently, if $f \in S_k(N)$ is a newform with CM by $E$, we must have $N_E \mid N$, where 
\[ N_E := |\Delta_E| \, \calN(\frakd_E) =
\begin{cases}
\Delta_E^2 &  \text{if } 2 \nmid \Delta_E \\
2 \Delta_E^2 & \text{if } 4 \parallel \Delta_E \\
4 \Delta_E^2 &  \text{if } 8 \parallel \Delta_E.
\end{cases} \]

In particular, if we want to determine the existence of a Gr\"ossencharacter $\psi$ of $E$ with value field $L=EK$ (resp.\ CM form $f$ with rationality field $K$), the minimal conductor (resp.\ level) in which it can occur is $\frakd_E$ (resp.\ $N_E$).  Under certain conditions, we can reduce to conductor $\frakd_E$ (resp.\ level $N_E$):

\begin{prop} \label{prop:mintwist}
Suppose $\psi$ is primitive with $\eta_\psi^2 = 1$ and $\eta_\psi^\Z = \chi_E$. Then there is a quadratic twist $\psi \otimes \chi$ which has conductor $\frakd_E$.  Moreover, either $\Delta_E$ is odd or $\Delta_E \equiv 0 \mod 8$.  
\end{prop}

This is essentially proven in \cite{schutt}, though it is formulated in the setting of $L = E$.

\begin{proof}
Say $\frakm = \frakp^e \frakm'$, where $\frakm'$ is coprime to $\frakp$. Let $p$ be the prime below $\frakp$.  Decompose $(\frako_E/\frakm)^\times = (\frako_E/\frakp^e)^\times \times (\frako_E/\frakm')^\times$.  Write $\eta = \eta_\psi$, and let $\eta_\frakp$ be the character of $(\frako_E/\frakp^e)^\times$ obtained by restricting $\eta$ to this factor. 

By \cref{lem:psi-extend}, $\eta$ is primitive.  Since $(\frako_E/\frakp^e)^\times$ is isomorphic to $(\frako_E/\frakp)^\times$ times a group of order $q^{e-1}$, where $q=\calN (\frakp)$, $\eta_\frakp$ must factor through $(\frako_E/\frakp)^\times$ if $q$ is odd.  Hence we may assume $e \le 1$ when $p$ is odd.  If $e=1$, primitivity means $\eta_\frakp$ is a (nontrivial) quadratic character.

Next suppose that $p$ is an odd prime such that $p \mid \calN(\frakm)$ and $p \nmid \Delta_E$.  We show we can twist away the $p$-part of the conductor.  

Consider the natural map $(\Z/p\Z)^\times \to (\frako_E/p\frako_E)^\times$, and let $\eta_p$ be the restriction of $\eta$ to $(\frako_E/p\frako_E)^\times$.  Then $\eta_p$ must be trivial on the image of this map.

If $p\frako_E = \frakp \bar \frakp$ is split in $E$, then identifying $(\frako_E/p\frako_E)^\times = (\frako_E/\frakp)^\times \times (\frako_E/\bar \frakp)^\times = (\Z/p\Z)^\times \times (\Z/p\Z)^\times$, we see $\eta_p(a,a) = 1$ for all $a \in (\Z/p\Z)^\times$.  Hence $\eta_p(a,b) = \chi_p(a)\chi_p(b)$, where $\chi_p$ denotes the unique quadratic Dirichlet character of conductor $p$.  (We remark this implies that if $\frakp \mid \frakm$, then $\bar \frakp \mid \frakm$ also.) Since we also have $\eta_p = \chi_p \circ \calN$ as characters of $(\frako_E/p\frako_E)^\times$, the twist $\psi \otimes \chi_p$ must have conductor $p^{-1} \frakm$.  

If $p$ is inert in $E$, then $(\frako_E/\frakp)^\times \simeq C_{p^2-1}$, and $\eta_\frakp$ factors through $C_{p^2-1}/C_{p-1}$.  Again, this must match with $\chi_p \circ \calN$, so that $\psi \otimes \chi_p$ has conductor $\frakp^{-1} \frakm$.

Finally we treat the case $p=2$.  

First suppose that $2 \mid \Delta_E$.  By \cref{lem:mincond}, we have that $e \ge 3$ if $4 \parallel \Delta_E$ and $e \ge 5$ if $8 \mid \Delta_E$.  Write $E = \Q(\sqrt{d_E})$ where $d_E$ is squarefree.
If $4 \parallel \Delta_E$, then $d_E \equiv 3 \mod 4$, so $\eta_\frakp(d_E) = -1$.  
But then $\eta_\frakp(\sqrt{d_E})$ would have order 4.  Hence $4 \parallel \Delta_E$ is impossible.
If $8 \mid \Delta_E$, then only $e=5$ is possible for $\eta_\frakp$ to be primitive and quadratic.  

Now assume that $p = 2 \nmid \Delta_E$.  Then $\eta_\frakp$ can only be quadratic and primitive if $e \le 3$.  

If $2 = \frakp \bar \frakp$ splits in $E$, then as in the $p$ odd case we may identify $\eta_\frakp = \eta_{\bar \frakp}$ as the same (possibly imprimitive) character of $(\Z/8\Z)^\times \simeq C_2^2$.  Viewing $\chi = \eta_\frakp$ as a Dirichlet character of conductor dividing 8, we see that $\psi \otimes \chi$ has conductor coprime to 2. 

If $2$ is inert in $E$, then $\eta_\frakp$ factors through $(\frako_E/8\frako_E)^\times$ modulo the image of $(\Z/8\Z)^\times$.  This quotient is isomorphic to $C_3 \times C_2^2$, so $\eta_\frakp$ factors through $C_2^2$.
Since $E_2/\Q_2$ is unramified, the norm induces a surjective map $(\frako_E/8\frako_E)^\times \to (\Z/8\Z)^\times$, and hence $\eta_\frakp$ must be of the form $\chi \circ \calN$ for one of the quadratic Dirichlet characters $\chi$ of conductor dividing 8.  For this $\chi$, one gets that $\psi \otimes \chi$ has conductor coprime to 2.
\end{proof}

\begin{cor} \label{cor:Lpsicount}
Let $\ell \ge 1$ be odd and $\calA^2_\ell(E)$ be the set of weight $\ell$ Gr\"ossencharacters $\psi$ of $E$ with quadratic modulus character $\eta_\psi$ such that $\eta_\psi^\Z = \chi_E$.  Assume $\mu_E = \{ \pm 1 \}$. If $\Cl(E)$ is cyclic or has exponent $2$, then the set $\{ L_\psi : \psi \in \calA^2_\ell(E) \}$ of value fields up to isomorphism has cardinality (i) $1$ if $\Delta_E$ is odd, (ii) $0$ if $4 \parallel \Delta_E$, and (iii) $1$ or $2$ if $8 \parallel \Delta_E$.
\end{cor}

\begin{proof}
There are no such $\psi$ when in case (ii), so assume we are in case (i) or (iii).  By \cref{prop:mintwist}, we may assume $\psi$ has modulus $\frakd_E$.  

We have $(\frako_E/\frakd_E)^\times \simeq (\Z/\Delta_E\Z)^\times$ when $\Delta_E$ is odd and $(\frako_E/\frakd_E)^\times \simeq (\Z/\Delta_E\Z)^\times \times C_4$ when $8 \mid \Delta_E$.  Hence there is 1 possibility for $\eta_\psi$ when $\Delta_E$ is odd and 2 possibilities for a quadratic $\eta_\psi$ when $8 \mid \Delta_E$ (the generator of the $C_4$ factor must map to $\pm 1$).  Now apply \cref{rem:Leta}.
\end{proof}

\section{Value fields for class number 1} \label{sec:h1}

Suppose $h_E = 1$, i.e., $-\Delta_E \in \{ 3,4,7,8,11,19,43,67,163 \}$.  Then for fixed $\frakm$ and $\ell$, the map $\psi \mapsto \eta_\psi$ defines a bijection of Gr\"ossencharacters $\psi$ with conductor $\frakm$ and weight $\ell$ and characters $\eta: (\frako_E/\frakm)^\times \to \C^\times$ such that $\eta(u) = u^{-\ell}$ for $u \in \frako_E^\times$.  Further, by \cref{lem:L0}, $L_\psi = \Q(\zeta_r) E$ where $r$ is the order of $\eta_\psi$.
Recall from \cref{sec:gros-def} that $\frako_E^\times \cap E^\frakm_1$ is trivial except in a few small cases.

We explicitly exhibit all fields of the form $L=\Q(\zeta_r)E$ with $d = [L:E] \le 3$ as value fields of Gr\"ossencharacters $\psi$ such that $\eta_\psi^\Z = \chi_E$.
It suffices to produce a character $\eta$ whose values generate $L$ and whose restriction to $\Z$ is $\chi_E$.  Then for any odd $\ell \ge 1$, $\eta$ determines a weight $\ell$ Gr\"ossencharacter such that $\eta_\psi = \eta$. 
We may assume $r$ is even, and construct $\eta$ as follows. 

From \cref{lem:mincond}, we need to work with conductors $\frakm$ which are multiples of $\frakd_E$.  We will usually choose conductors of the form $\frakm =  \frakp^e \frakd_E$ where $\frakp$ is coprime to $\frakd_E$.  Then it suffices to choose $\frakp^e$ such that $(\frako_E/\frakp^e)^\times / (\Z/p^e\Z)^\times$ (meaning the quotient by the image of the natural map\footnote{For simplicity, from now on if $n \in \Z$ and $\frakn$ is an ideal in $\frako_E$ such that the natural map $(\Z/n\Z)^\times$ to $(\frako_E/\frakn)^\times$ is an injection, we identify $(\Z/n\Z)^\times$ with its image in $(\frako_E/\frakn)^\times$.}) has a character $\lambda$ of order $r$, or $\frac r2$ if $r \equiv 2 \mod 4$.  Then prescribing $\eta$ to be $\chi_E$ mod $\Delta_E$ and $\lambda$ mod $\frakp^e$ makes $\eta$ a character of order $r$ and conductor $\frakm$ whose restriction to $\Z$ is $\chi_E$.

\subsection{$d=1$}
First we explain how to construct $\eta$ with values in $L=E$.  This typically means we can take $\eta$ to be quadratic, and usually we may take $\frakm = \frakd_E$.

If $\Delta_E < -4$ is odd, we may take $\frakm = \frakd_E = \sqrt{\Delta_E} \frako_E$ and then $\frako_E^\times \cap E^\frakm_1 = \{ 1 \}$.  Since $(\frako_E/\frakm)^\times \simeq (\Z/\Delta_E \Z)^\times$, let $\eta = \chi_E$.

If $E = \Q(i)$, take $\frakm = \frakd_E = 2(1+i)\frako_E$.  Then $\frako_E^\times \cap E^\frakm_1 = \{ 1 \}$, and $(\frako_E/\frakm)^\times \simeq \{ \pm 1, \pm i \}$.  Then $\eta(u) = u^{-\ell}$ is an order 4 character of $(\frako_E/\frakm)^\times$, and it restricts to the nontrivial quadratic character on $(\Z/4 \Z)^\times \simeq \{ \pm 1 \}$.

If $E = \Q(\zeta_3)$, take $\frakm = \frakd_E^2 = 3 \frako_E$.  Here we do not take $\frakm=\frakd_E$ because $\frako_E^\times \cap E^{\frakd_E}_1 \ne \{ 1 \}$.  However, $\frako_E^\times \cap E^\frakm_1 = \{ 1 \}$, and $(\frako_E/\frakm)^\times \simeq \langle \zeta_6 \rangle$.  Then $\eta(u) = u^{-\ell}$ is a character of $(\frako_E/\frakm)^\times$ which restricts to the nontrivial quadratic character on $(\Z/3\Z)^\times$.

If $E = \Q(\sqrt{-2})$, take $\frakm = \frakd_E = 4\sqrt{-2} \frako_E$.  Then $(\frako_E/\frakm)^\times \simeq (\Z/8\Z)^\times \times C_4$, and we may take $\eta$ to be either quadratic character restricting to $\chi_E$ on $(\Z/8\Z)^\times$.

In all of these cases, for any odd $\ell \ge 1$, $\eta$ determines a Gr\"ossencharacter $\psi$ of $E$ with conductor $\frakm$ and weight $\ell$ such that $\eta_\psi^\Z = \chi_E$ and $L_\psi = E$.  In particular, one gets a CM form $f_\psi \in S_{\ell + 1}(N)$ with rationality field $K = \Q$, where $N = N_E$ or $27$ according to whether $\Delta_E \ne -3$ or $\Delta_E = -3$.

\subsection{$d=2$} \label{sec:h1d2}

Now we construct $\eta$ whose values generate a quadratic extension $L = \Q(\zeta_r) E$ of $E$.

Assume first that $\Delta_E < -4$.  Then $\mu_E = \{ \pm 1 \}$, and the only quadratic extensions of the form $L=\Q(\zeta_r)E$ are $E(i)$ and $E(\zeta_3)$.  

We first treat $r=4$, i.e., $L = E(i)$.  

Except in the case of $\Delta_E = -8$ and $\Delta_E = -11$, the prime 3 is inert in $E$, and thus we take $\frakm = 3\frakd_E$.  Then $(\frako_E/\frakm)^\times \simeq (\frako_E/3\frako_E)^\times \times 
(\frako_E/\frakd_E)^\times$.  We define $\eta$ to be one of the order 4 characters on the $(\frako_E/3\frako_E)^\times/(\Z/3\Z)^\times \simeq C_4$ component and $\chi_E$ on $(\frako_E/\frakd_E)^\times \simeq (\Z/\Delta_E \Z)^\times$ component.  

When $\Delta_E = -8$, we can take $\frakm = \frakd_E$ since $(\frako_E/\frakm)^\times \simeq (\Z/8\Z)^\times \times C_4$, and let $\eta$ be order 4 on the $C_4$ component.

When $\Delta_E = -11$ (and also when $\Delta_E = -19$), 5 is split in $E$ and we can take $\frakm = 5\frakd_E$.  Then we take $\eta$ to be of the form $\lambda(a,b) = \lambda(a) \lambda(b)^{-1}$ on the factor $(\frako_E/5\frako_E)^\times \simeq (\Z/5\Z)^\times \times (\Z/5\Z)^\times$, where $\lambda$ is an order 4 character.  

Hence for each $\Delta_E < -4$ we get a Gr\"ossencharacter $\psi$ of conductor $\frakm$ and any odd weight $\ell$ with value field $L = E(i)$.  Then the associated CM form $f_\psi$ has rationality field $K = \Q(\sqrt{|\Delta_E|})$ and level $9N_E$, $N_E$, or $25N_E$, according to the 3 cases above.

Next we treat $r=6$, i.e., $L = E(\zeta_3)$.

Except for $\Delta_E = -7, -8$, one sees that 2 is inert in $E$.  Thus there exists an order 6 character $\eta$ mod $\frakm = 2 \frakd_E$ which is compatible with $\chi_E$.

Except for $\Delta_E = -11, -19$, the prime 5 is inert in $E$.  Hence there exists an order 6 character $\eta$ mod $\frakm = 5 \frakd_E$ which is compatible with $\chi_E$.

In these situations we get a Gr\"ossencharacter $\psi$ of conductor $\frakm$ and any odd weight $\ell$ with value field $L = E(\zeta_3)$.  Then the associated CM form $f_\psi$ has rationality field $K = \Q(\sqrt{3|\Delta_E|})$ and level $4N_E$ or $25N_E$.

Finally let us consider $\Delta_E = -4, -3$.

If $E = \Q(i)$, then one can have $L = \Q(\zeta_8)$ or $L = \Q(\zeta_{12})$.  Since $7$ and $11$ are inert in $E$, we can achieve each of these value fields  with conductors $\frakm = 7 \frakd_E$ and $\frakm = 11 \frakd_E$, respectively.  This produces CM forms in even weights $k \ge 2$ of levels $2^5 \cdot 7^2$ and $2^5 \cdot 11^2$ with rationality fields $K = \Q(\sqrt 2)$ and $K = \Q(\sqrt 3)$.

If $E = \Q(\zeta_3)$, then one can have $L = \Q(\zeta_{12})$.  Again, 11 is inert in $E$, so we obtain this value field in conductor $\frakm = 11 \frakd_E$.  This yields a CM form in even weights $k \ge 2$ of level $3^2 \cdot 11^2$ with rationality field $\Q(\sqrt 3)$.

\subsection{$d=3$} \label{sec:h1d3}
Now we construct $\eta$ whose values generate a cubic extension $L = \Q(\zeta_r) E$ of $E$.  There are no degree 3 cyclotomic extensions of $\Q$, so the only possibility is $E \subset \Q(\zeta_r)$ and $\Q(\zeta_r)$ has degree 6.
Thus either $E = \Q(\zeta_3)$ and $L = \Q(\zeta_9)$ or $E = \Q(\sqrt{-7})$ and $L = \Q(\zeta_7)$.

Suppose $E = \Q(\zeta_3)$ and $L = \Q(\zeta_9)$.
Take $\frakm = \frakd_E^6 = 27 \frako_E$.  From the Type 4 case of \cite{ranum}, it follows that
$(\frako_E/\frakm)^\times \simeq C_6 \times C_9 \times C_9$, where we can take for generators of each component $\zeta_6$, $4$ and $1+3 \sqrt{-3}$.  For odd $\ell$, define $\eta$ to be the character such that $\eta(\zeta_6) = \zeta_6^{-\ell}$, $\eta(4) = 1$, and $\eta(1+3\sqrt{-3}) = \zeta_9$.
This gives a Gr\"ossencharacter $\psi$ of weight $\ell$ with $\eta_\psi^\Z = \chi_E$ and $L_\psi = L$.  Then $f_\psi \in S_{\ell + 1}(3^7)$ and has rationality field $K = \Q(\zeta_9)^+$.

Suppose $E = \Q(\sqrt{-7})$ and $L = \Q(\zeta_7)$.  We can take $\frakm = \frakd_E^2 = 7 \frako_E$.  Then $(\frako_E/\frakm)^\times \simeq (\Z/7\Z)^\times \times C_7$.  Define $\eta$ to be the quadratic character on $(\Z/7\Z)^\times$ and order 7 on $C_7$.  This gives a Gr\"ossencharacter $\psi$ for any odd weight $\ell$ with $\eta_\psi^\Z = \chi_E$ and $L_\psi = L$.  One has $f_\psi \in S_{\ell + 1}(7^3)$ with rationality field $K = \Q(\zeta_7)^+$.

\section{Value fields with quadratic modular characters}
\label{sec:quadmod}

Here we will determine value fields $L$ of Gr\"ossencharacters $\psi$ of $E$ such that $\eta_\psi$ is quadratic and $\eta_\psi^\Z = \chi_E$, when $E$ has exponent 2 or 3.  
By \cref{prop:mintwist}, we may assume $\Delta_E \not \equiv 4 \mod 8$ and the conductor of $\psi$ is $\frakm = \frakd_E$.  Hence the associated CM form $f_\psi$ will have level $N_E$.

Suppose $d = [L:E]$ is prime and $\psi$ has odd weight $\ell \ge 1$.
Let $\frakt_1, \dots \frakt_g$ be representatives for a set of generators for $\Cl(E)$ which are coprime to $\frakd_E$.  Let $n_i$ be the order of $\frakt_i$ in $\Cl(E)$ and write $\frakt_i^{n_i} = \theta_i \frako_E$.  By \cref{prop:vf}, for each $1 \le i \le g$ we must have $L _\psi = E((\pm \theta_i)^{\ell/n_i})$ for some choice of $\pm$ (which may depend upon $i$).

Assuming $\ell$ is coprime to $h_E$ we have the following necessary condition:

\begin{enumerate}[label={(Q\arabic*)}]
\item \label{cond:Q2}
if $\Cl(E)$ is not cyclic, there exists a collection of signs $\eps_1, \dots, \eps_g \in \{ \pm 1 \}$ such that $E((\eps_i \theta_i^\ell)^{1/n_i}) = E((\eps_j \theta_j^\ell)^{1/n_j})$ for all $1 \le i < j \le g$. 
\end{enumerate}

\subsection{Exponent 2 fields} \label{sec:quadexp2}
There are 56 known imaginary quadratic fields $E$ with (class group) exponent 2, with the discriminant of largest absolute value being $-5460$, and this list is complete under ERH \cite{weinberger}.  Of these, 18 have class group $C_2$, 24 have class group $C_2^2$, 13 have class group $C_2^3$, and 1 has class group $C_2^4$.

For those with class number $h_E > 2$, we checked in Sage that \ref{cond:Q2} never holds for $\ell$ odd. (Note that we may assume $\ell = 1$ for this.)

Among those with class number 2 (so \ref{cond:Q2} is vacuous), there are 15 with $\Delta_E \not \equiv 4 \mod 8$, namely those with
\[ -\Delta_E \in \{ 15, 24, 35, 40, 51, 88, 91, 115, 123, 187, 232, 235, 267, 403, 427 \}. \]

Assume $E$ is one of these 15 fields.  Since $g=1$, write $\theta = \theta_1$.  Then $L = E(\sqrt{\eta_\psi(\theta) \theta})$, so to determine the value fields for each $E$ we just compute a representative $\theta$ and $\eta(\theta)$ for each character $\eta$ of $(\frako_E/\frakd_E)^\times$ which restricts to $\chi_E$ on $(\Z/\Delta_E \Z)^\times$.
(Recall from \cref{cor:Lpsicount} that there is a unique value field $L$ when $\Delta_E$ is odd, and at most two value fields when $8 \mid \Delta_E$; we will see that there is a unique value field in all cases.)

When $\Delta_E$ is odd, $(\Z/\Delta_E\Z)^\times$ surjects onto $(\frako_E/\frakd_E)^\times$ and thus $\eta_\psi(\theta)$ is determined by $\chi_E$.  We compute the value fields $L = L_\psi$ for these $\Delta_E$ and record them in \cref{tab:quadodd}.  Since $L=EK$ is specified by $E$ and its real quadratic field $K$, we just record $\Delta_K$ in this table.

\begin{table} 
\centering
\begin{tabular}{c||c|c|c|c|c|c|c|c|c|c|c} $\Delta_E$ & $-15$ & $-35$ & $-51$ & $-91$ & $-115$ & $-123$ & $-187$ & $-235$ & $-267$ & $-403$ & $-427$ \\
\hline
$\Delta_K$ & $5$   & $5$   & $17$  & $13$  & $5$    & $41$   & $17$   & $5$    & $89$  & $13$ & $61$ \\
\end{tabular}
\caption{Value fields $L=EK$ when $\Delta_E$ is odd}
\label{tab:quadodd}
\end{table}

Now suppose $\Delta_E \equiv 0 \mod 8$.  Write $\Delta_E = -8 q$, where $q \in \{ 3, 5, 11, 29 \}$, and let $\frakq$ be the prime above $q$.  Now the image of $\theta$ in $(\frako_E/\frakd_E)^\times$ may not be in the image of $(\Z/\Delta_E\Z)^\times$.  Indeed, computation shows that in each case we can choose $\theta$ such that $\theta \equiv 2 \mod \frakq$ and $\theta \equiv -q \eps^2 \mod \frakp_2^5$, where $\eps = 1 + \sqrt{-2q}$ has order 4 mod $\frakp_2^5$.  Since $\eta$ is quadratic, one has that $\eta_2(\theta) = \eta_2(-q)$ where $\eta_2$ is $\eta$ composed with the projection to $(\frako_E/\frakp_2^5)^\times$.  Thus $\eta(\theta) = \chi_E(t)$, where $t \in (\Z/\Delta_E\Z)^\times$ such that $t \equiv 2 \mod q$ and $t \equiv -q \mod 8$.  (Hence the two possible $\eta$'s must yield the same value field $L$.)

Explicitly for $q = 3, 5, 11, 29$, we can respectively take $t = 5, 27,  13 ,147$ and find that $\eta(\theta) = -{-1 \leg q}$.  We tabulate the value fields $L=EK$ via discriminants in \cref{tab:quadeven}.

\begin{table}
\centering
\begin{tabular}{c||c|c|c|c}  
$\Delta_E$ & $-24$ & $-40$ & $-88$ & $-232$ \\
\hline
$\Delta_K$ & $8$ & $5$ & $8$ & $29$  \\
\end{tabular}
\caption{Value fields $L=EK$ when $\Delta_E$ is even}
\label{tab:quadeven}
\end{table}

\subsection{Exponent 3 fields} \label{sec:qmce3}
There are 17 known imaginary quadratic fields $E$ with exponent 3, and there are no others if ERH holds \cite{EKN}.  All of these fields have negative prime discriminant, so $\Delta_E$ is odd.  One of them, $\Q(\sqrt{-4027})$, has class group $C_3^2$, and one checks this field does not satisfy \ref{cond:Q2} (assuming $3 \nmid \ell$, so one only needs to check $\ell = 1, 2$).  

  All the others have class group $C_3$, and their discriminants satisfy
\[ -\Delta_E \in \{ 23, 31, 59, 83, 107, 139, 211, 283, 307, 331, 379, 499, 547, 643, 883, 907 \}. \]
For each of these class number 3 fields, there is a unique value field $L$ that occurs by \cref{cor:Lpsicount}, which we can construct in a similar way to the $\Delta_E$ odd case of \cref{sec:quadexp2}.  Namely, $L$ is of the form $E((\eta(\theta) \theta^\ell)^{1/3})$, and it only depends on $\ell \mod 3$.  In  \cref{tab:quade3} we indicate the value fields $L=EK$ for $\ell \equiv 1 \mod 6$, where $K$ is a totally real cubic field with the listed discriminant.  In all of these cases the discriminant uniquely identifies $K$ up to isomorphism.

\begin{table}
\centering
\begin{tabular}{c||c|c|c|c|c|c|c|c} 
$\Delta_E$ & $-23$ & $-31$ & $-59$ & $-83$ & $-107$ & $-139$ & $-211$ & $-283$ 
\\
\hline
$\Delta_K$ & $621$ & $837$ & $1593$ & $2241$ & $321$ & $3753$ & $5697$ & $7641$ \\
\end{tabular}

\vspace{8pt}
\begin{tabular}{c||c|c|c|c|c|c|c|c} 
$\Delta_E$ & $-307$ & $-331$ & $-379$ & $-499$ & $-547$ & $-643$ & $-883$ & $-907$ \\
\hline
$\Delta_K$ & $8289$ & $993$ & $10233$ & $13473$ & $14769$ & $1929$ & $23841$ & $24489$
\end{tabular}
\caption{Value fields $L=EK$ for $\ell \equiv 1 \mod 6$}
\label{tab:quade3}
\end{table}

\section{Value fields with higher order modular characters}
\label{sec:last}

Lastly, we consider $\psi$ such that $\eta_\psi$ has order $r > 2$ and $\eta_\psi^\Z = \chi_E$.  Combining the results below with those in the previous two sections will prove \cref{thm2}.

Assume $h_E > 1$, so that $L_0 = E(\zeta_r) \ne E$.  If $d = [L:E]$ is prime then $L=L_0$, and one needs the following:


\begin{enumerate}[label={(R\arabic*)}]
\item \label{eq:rootcond}
For each $1 \le i \le g$, there exists $\zeta \in \mu_L$ such that $(\zeta \theta_i^\ell)^{1/n_i} \in E(\zeta_r)$.
\end{enumerate}

\subsection{Exponent 2 fields}

Suppose $E$ has exponent 2 and $L = L_0 = E(\zeta_r)$ has degree 2 over $E$.  Then we can either take $r=4$ or $r=6$.

We checked in Sage that \ref{eq:rootcond} holds exactly in the following situations:

\begin{enumerate}[label=(\roman*)]
\item
$r = 4$, $-\Delta_E \in \{ 20, 24, 40, 52, 88, 148, 232 \}$, and $\zeta$ is a primitive 4th root of unity; or

\item
$r = 6$: $-\Delta_E \in \{ 15, 24, 51, 123, 267 \}$, and $\pm \zeta$ is a primitive 3rd root of unity.
\end{enumerate}

All of these are class number 2 fields.  Take a representative $\frakt$ for the nontrivial element of $\Cl(E)$ which is coprime to $\Delta_E$.  Say $\frakt^2 = \theta \frako_E$.

First suppose we are in case (i).  Then 2 is ramified in $E$, and we let $\frakp_2$ denote the prime above 2.  To see whether a given $E$ in (i) yields a $\psi$ as above amounts to determining whether there exists a modulus $\frakm$ with character $\eta$ which agrees with $\chi_E$ and satisfies $\eta(\theta) = \pm i$.

\subsubsection{When $r = 4$, $\Delta_E \equiv 4 \mod 8$}
If $\Delta_E \equiv 4 \mod 8$, then we compute in each case that $\theta^2 \equiv a \mod \frakd_E$, where $a \in (\Z/\Delta_E\Z)^\times$ is an element on which $\chi_E$ is $-1$.  Specifically, for $\Delta_E = -20, -52, -148$, we take $\theta$ to be an element of $E - \Q$ of norm $9, 49, 361$ and find $a = -1, 23, -9$, respectively.  Hence for these $\Delta_E$, either character $\eta$ of $(\frako_E/\frakd_E)^\times$ which restricts to $\chi_E$ on $(\Z/\Delta_E \Z)^\times$ sends $\theta$ to a primitive 4th root of unity, and thus it gives a Gr\"ossencharacter $\psi$ for any odd weight $\ell$ such that $L_\psi = E(i)$.  This gives CM forms $f \in S_{\ell + 1}(2 \Delta_E^2)$ with $K_f = \Q(\sqrt{|\Delta_E|})$.

\subsubsection{When $r = 4$, $\Delta_E \equiv 0 \mod 8$}
Assume now $\Delta_E \equiv 0 \mod 8$.  We claim there is no modulus $\frakm$ with an order 4 character $\eta : (\frako_E/\frakm)^\times \to \C^\times$ such that $\eta$ restricted to $(\Z/\calN(\frakm))^\times$ is $\chi_E$ and $\eta(\theta) = \pm i$.  (Tacitly, when necessary we replace $\frakt$ and $\theta$ with representatives that are coprime to $\frakm$.)

Suppose that such a $\frakm$ and $\eta$ exist.  We may assume $\eta$ is primitive.  Then there exists a prime $\frakp \mid \frakm$ such that $\eta_\frakp(\theta) = \pm i$, where $\eta_\frakp$ is $\eta$ composed with the projection $(\frako_E/\frakm)^\times \to (\frako_E/\frakp^e)^\times$ and $e = v_\frakp(\frakm)$.  Let $\frakp$ be such a prime and $p$ be the prime below $\frakp$. Necessarily, $\theta$ is a nonsquare mod $\frakp^e$, because if it were the square of some $\beta$ mod $\frakp^e$, then $\eta_\frakp(\beta)$ and hence $\eta$ would have to have order strictly larger than 4.

First suppose that $\frakp = \frakp_2$ is the prime above 2.  We find in Sage that the minimal $e$ such that $\theta \mod \frakp_2^e$ has order 4 is $e=6$.  However, in this case, $(\frako_E/\frakp_2^6)^\times \simeq (\Z/8\Z)^\times \times C_8$.  Hence $\eta(\theta)\in \{ \pm i \}$ would imply $\eta$ has order $\ge 8$, and thus $\frakp = \frakp_2$ is impossible.

So assume $p$ is odd.  Then, as in the proof of \cref{prop:mintwist}, primitivity implies $e=1$.  Note that $\calN(\theta) = \calN(\frakt)^2$ is a square in $(\Z/p\Z)^\times$.

If $p$ ramifies in $E$, then $(\Z/p\Z)^\times \to (\frako_E/\frakp)^\times$ is surjective, whence $\eta_\frakp$ cannot have order 4 on the target but not the source of this map.  
Suppose $(p) = \frakp \bar \frakp$ is split in $E$, and let $\eta_p$ be the character of $(\frako_E/p\frako_E)^\times$ as in the proof of \cref{prop:mintwist}.  Then $\eta_p(\theta) := \eta_\frakp(\theta) \eta_{\bar \frakp}(\theta) = \eta_\frakp(\theta \bar \theta^{-1}) = \eta_\frakp(\theta \bar \theta) \eta_\frakp(\bar \theta)^{-2}  = (\eta_\frakp(\calN(\frakt))/\eta_\frakp(\bar \theta))^2 \in \{ \pm 1 \}$.  
In other words, ramified and split $\frakp \mid \frakm$ cannot account for $\eta(\theta) = \pm i$.

Hence $\eta_\frakp(\theta) = \pm i$ for an odd $p$ inert in $E$.
Since the norm map from $(\frako_E/\frakp)^\times \to (\Z/p\Z)^\times$ sends the subgroup of squares to the subgroup of squares, this means that $\theta$ is a square mod $\frakp$, which is our desired contradiction.

\subsubsection{When $r = 6$}
Now consider case (ii).  Note that 3 is ramified in $E$.  In each case, $\theta$ generates $(\frako_E/3\frako_E)^\times \simeq C_6$.  Let $\frakp_3$ be the prime in $E$ above 3.    Then for any odd $\ell$, there is a Gr\"ossencharacter $\psi$ mod $\frakm := \frakp_3 \frakd_E$ such that $L_\psi = E(\zeta_3)$ and $\eta_\psi^\Z = \chi_E$.  The associated CM form $f_\psi \in S_{\ell + 1}(3 N_E)$ has rationality field $K = \Q(\sqrt{\frac{|\Delta_E|}3})$.  For each of these $E$, this does not produce any value/rationality fields different from the case of quadratic modulus character in \cref{sec:quadexp2}, though we now get forms of level $3N_E$ instead of level $N_E$.

\subsection{Exponent 3 fields}
Suppose $E$ has exponent 3 and $L = L_0 = E(\zeta_r)$ has degree 3 over $E$.  Then $L$ must be a cyclotomic field of degree 6, and $E$ is the quadratic subfield.  Namely $L = \Q(\zeta_7)$ and $E = \Q(\sqrt{-7})$ or $L = \Q(\zeta_9)$ and $E = \Q(\sqrt{-3})$.  In both cases $h_E = 1$, violating our assumption.  (We already constructed such Gr\"ossencharacters in \cref{sec:h1d3}.)

\appendix
\renewcommand{\thesection}{A}

\section{Structure of dyadic modular quotients}
\label[app]{sec:append}

Let $E/\Q$ be a quadratic field, and $\frakp$ be a prime above $2$.  Here we recall some structure theory about $(\frako_E/\frakp^n)^\times$.  If $2$ splits in $E$, then $(\frako_E/\frakp^n)^\times \simeq (\Z/2^n \Z)^\times \simeq C_2 \times C_{2^{n-2}}$ for $n \ge 2$.  So assume $2$ is either inert or ramified in $E$.

Consider a local quadratic extension  $k/\Q_2$ with integer ring $\frako_k$ and prime ideal $\frakp_k$.  Write $\fraku_n = (1 + \frakp_k^n) \cap \frako_k^\times$ for the $n$-th unit group of $k$.  Put $f = f(k/\Q_2)$ and $e = e(k/\Q_2)$.  It is well known that when $k = E_{\frakp_2}$,
\[ G := (\frako_E/\frakp^n)^\times = (\frako_k/\frakp_k^n)^\times \simeq C_{2^f-1} \times \fraku_1/\fraku_{n}. \]

First we have the following facts about the 2-rank of $G$ (e.g., see \cite{nakagoshi}, which also treats higher-degree fields).

\begin{enumerate}[label=(\roman*)]
\item Suppose $n < 2e + 1$.  Then the 2-rank of $G$ is $\lceil \frac {n-1}2 \rceil f$.

If $2$ is inert in $E$, then $G$ has 2-rank 0 or 2 when $n=1$ or $n=2$, respectively.  Hence $G \simeq C_3$ or $G \simeq C_3 \times C_2^2$ in these cases.

If $2$ ramifies in $E$, then $G$ has 2-rank $0, 1, 1, 2$ for $n = 1, 2, 3, 4$, respectively. Necessarily $G$ is isomorphic to $C_1$, $C_2$, $C_4$, or $C_2 \times C_4$.

\item Suppose $n \ge 2e + 1$.  Then the 2-rank is $ef+1=3$.

If $2$ is inert in $E$, this means $G$ has 2-rank 3 when $n \ge 3$.  Necessarily $G \simeq C_3 \times C_2^2 \times C_4$ when $n=3$.

If $2$ ramifies in $E$, then $G$ has 2-rank 3 when $n \ge 5$.  Necessarily $G \simeq C_2^2 \times C_4$ when $n=5$.
\end{enumerate}

The following more precise information was obtained in \cite{ranum}; these cases respectively correspond to Type 9, Types 6--7, and Type 5 in \cite{ranum}.  Write $E = \Q(\sqrt d)$ where $d$ is squarefree. 

\begin{enumerate}[label=(\roman*),resume]

\item Suppose 2 is inert in $E$. Then
$G = (\frako_E/2^n \frako_E)^\times \simeq C_3 \times \langle -1 \rangle \times C_{2^{n-1}} \times C_{2^{n-2}}$ for $n \ge 2$.  Further, 5 generates a subgroup of index 2 of the third factor, and $3+2 \sqrt d$ generates the last factor.  The image of $(\Z/2^n\Z)^\times$ in $(\frako_E/\frakp_2^n)^\times$ is the subgroup generated by $-1$ and $5$.

\item Suppose $4 \parallel \Delta_E$.  The natural map $(\Z/4\Z)^\times \to (\frako_E/\frakp^n)^\times$ is injective if and only if $n \ge 3$.   

The precise structure of $G$ depends on $d$ mod 8 and whether $n = 2, 3, 4$ or $n \ge 5$.  However, we do not need more than is already stated in (i) and (ii) above, so we simply remark that for any $n\ge 2$ one has $G \simeq C_{2^j} \times C_{2^a} \times C_{2^b}$ for some $a, b \ge 0$ such that $a+b = n-j-1$, where $j \in \{ 1, 2 \}$.  (If $n \ge 5$, then $|a - b| \le 1$.)

\item Suppose $8 \parallel \Delta_E$.  The natural map $(\Z/8\Z)^\times \to G$ is injective if and only if $n \ge 5$.  Moreover, $G \simeq \langle -1 \rangle \times C_{2^{r-2}}  \times C_{2^s}$, where $s = \lfloor \frac n2 \rfloor$ and $r = \lceil \frac n2 \rceil$ if $n \ge 4$.  In this case 5 generates the $C_{2^{r-2}}$ factor, and $1 + \sqrt d$ generates the $C_{2^s}$ factor.

\end{enumerate}

%
%

\end{document}